\documentclass[a4paper, 11pt]{amsart} 
\usepackage{amssymb, latexsym, amsrefs, amsmath, amsthm, color, 
cancel}
\usepackage{a4wide}
\usepackage[mathscr]{euscript} \DeclareMathAlphabet{\euls}{U}{eus}{m}{n}

\newtheorem{theorem}[equation]{Theorem}

\newtheorem{lemma}[equation]{Lemma}

\newtheorem{corollary}[equation]{Corollary}

\theoremstyle{definition}
\newtheorem{definition}[equation]{Definition}

\newtheorem{question}[equation]{Question}
\newtheorem{questions}[equation]{Questions}

\theoremstyle{remark}
\newtheorem{remark}[equation]{Remark}
\newtheorem{remarks}[equation]{Remarks}
\newtheorem{notation}[equation]{Notation}

\numberwithin{equation}{section}

\newcommand{\e}{\varepsilon}
\newcommand{\g}{\mathfrak{g}}

\begin{document}
\title[$AS$-Gorenstein Hopf algebras]{Report on $AS$-Gorenstein Hopf algebras}

\author[K. A. Brown]{Ken A. Brown}
\address{School of Mathematics and Statistics\\
University of Glasgow\\ Glasgow G12 8QQ\\
Scotland}
\email{ken.brown@glasgow.ac.uk}


\date{\today}
 
\begin{abstract} This is a review of progress on the question whether noetherian Hopf algebras always have finite injective dimension and related good homological properties. As well as discussing in detail the main results giving positive answers for particular classes of Hopf algebras, some consequences of such positive answers are also described. Full definitions and references are included, also sketches of some proofs.  A considerable number of open questions are listed, additional to the original question, which itself remains open after 30 years. 
 \end{abstract}
 \maketitle
\tableofcontents

\bigskip

\section{Introduction}\label{intro}

The primary purpose of this paper is to report on the current status of the following question, which is a more precise version of one first asked in \cite[$\S$3.1, Question A and remark following]{Br97} and \cite[$\S$1.14]{BG97}\footnote{This question is sometimes now referred to as the ``Brown-Goodearl conjecture", although it is simply stated as a question in both these papers}. The terms used in the question are explained in Definition \ref{xxdefn2.1}.

\begin{question}\label{xxqn1.1} Let $\mathbb{K}$ be a field $H$ be a noetherian Hopf $\mathbb{K}$-algebra. Is $H$ Auslander-Gorenstein and $AS$-Gorenstein?
\end{question}

 A related question, also discussed below, is: when is  $H$ $GK$-Cohen Macaulay? As well as reviewing current answers to these questions for various classes of Hopf algebras, we also give an account of the consequences flowing from positive answers. A number of open questions are listed throughout.

The layout is as follows. $\S$\ref{xxsec2} contains the basic definitions required and some relations between them. The heart of $\S$\ref{xxsec3} is Theorem \ref{xxthm3.3}, providing a list of classes of Hopf algebras for which positive answers to Question \ref{xxqn1.1} have been obtained, with references to the original papers and indications of some of the proofs. Then Theorem \ref{xxthm3.9} focusses on the class of affine noetherian Hopf algebras satisfying a polynomial identity, describing structural consequences which flow from the homological properties. As well as being significant for PI Hopf algebras themselves, these consequences may suggest possible features possessed by, but not yet proven for, other classes of noetherian Hopf algebras. In $\S$\ref{subsectlow} a similar review is given for affine noetherian Hopf algebras of Gelfand-Kirillov dimension at most 2. Finally, $\S$\ref{xxsec4} surveys some applications of the positive homological results which have so far been obtained.

\bigskip

\begin{notation}\label{xxnotate} Throughout, $\mathbb{K}$ denotes a field; further conditions will be imposed on $\mathbb{K}$ as needed. Given a $\mathbb{K}$-algebra $A$ and a left or right $A$-module $M$ the \emph{Gelfand-Kirillov dimension} of $M$ is denoted by $\mathrm{GKdim} (M)$; our basic reference for GK-dimension is \cite{KL00}. The projective dimension of a (right or left) $A$-module $M$ is denoted  $\mathrm{prdim}(M)$, and the global [resp. injective] dimension of an algebra $A$ is denoted by $\mathrm{gldim}(A)$ [resp. $\mathrm{injdim}(A)]$. A Hopf $\mathbb{K}$-algebra $H$ is assigned the coproduct $\Delta$, antipode $S$ and counit $\varepsilon$, with $\Delta (h)$ denoted by $\sum h_1 \otimes h_2$ for $h \in H$. The symbols ${}_{\e}\mathbb{K}$ and $\mathbb{K}_{\e}$ denote (respectively) the trivial left and right $H$-modules. The unadorned tensor product symbol $\otimes$ always means $\otimes_{\mathbb{K}}$.
\end{notation}

\bigskip

\section{Basic definitions and their interplay}\label{xxsec2} 

We'll focus thoughout on noetherian Hopf $\mathbb{K}$-algebras. It will frequently be necessary to assume that such an algebra is also \emph{affine}, meaning finitely generated as a $\mathbb{K}$-algebra, but it is a moot point whether this additional hypothesis is necessary. That is, it is currently not known whether the answer to the following question is ``yes''.

\begin{question}\label{affineqn} Is every noetherian Hopf $\mathbb{K}$-algebra affine?
\end{question} 

The answer is ``yes'' for a Hopf algebra which is pointed, commutative or cocommutative, thanks to Molnar \cite{Mo75} for the last 2 cases and to Goodearl and Zhang \cite{GZ17} for the first\footnote{In 
\cite{GZ17} there is a global assumption that the ground field is algebraically closed of characteristic 0, but neither hypothesis is needed for their proof of this result - see \cite[Theorem 6.4(3)]{GZ17}.}.

\begin{definition}\label{xxdefn2.1} Let $A$ be a noetherian $\mathbb{K}$-algebra.
\begin{enumerate}
\item[(i)]  If the injective dimensions of ${}_A A$ and $A_A$ are finite then $A$ is said to have \emph{finite injective dimension}. In this case these integers are equal by \cite{Za69}, and are denoted by $\mathrm{injdim}(A)$. 
\item[(ii)] $A$ is \emph{regular} if it has finite global dimension; that is, $\mathrm{gldim}(A) < \infty$. Right global dimension always equals left global dimension \cite[Exercise
4.1.1]{We94}; and when $\mathrm{gldim}(A)$ is finite, $\mathrm{gldim}(A) = \mathrm{injdim}(A)$.
\item[(iii)] $A$ is \emph{Artin–Schelter Gorenstein}, usually abbreviated to $AS$–\emph{Gorenstein}, if
\begin{itemize}
\item[(AS1)] $\mathrm{injdim}({}_A A) = d < \infty$;
\item[(AS2)] for every simple left $A$-module $V$ with $\mathrm{dim}_{\mathbb{K}}(V) < \infty$, $\mathrm{Ext}^d_A (V, {}_A A)$ is a non-zero finite dimensional simple (right) $A$-module, and $\mathrm{Ext}_A^i(V, {}_A A) = 0$ for all $i \neq d$;
\item[(AS3)] the right $A$-module versions of (AS1) and (AS2) hold.
\end{itemize}
\item[(iv)]  If $A$ is $AS$-Gorenstein and $\mathrm{gldim}(A) < \infty$, then $A$ is called \emph{Artin–Schelter regular}, usually shortened to $AS$-\emph{regular}.
\item[(v)] The \emph{grade} of a left $A$-module $M$ is
$$ j_A(M) \; := \; \mathrm{inf}\{n\, : \mathrm{Ext}_A^n (M,{}_A A) \neq 0 \} \; \in \; \mathbb{N} \cup \infty.$$
An analogous definition applies to right $A$-modules.
\item[(vi)] $A$ is \emph{GK-Cohen Macaulay} if $\mathrm{GKdim}(A) < \infty$ and for each non-zero finitely generated left or right $A$-module $M$
\begin{equation}\label{xxeqn2.2} \mathrm{GKdim}(M) \, + \, j_A (M) \; = \; \mathrm{GKdim}(A). 
\end{equation}
\item[(vii)] A left or right $A$-module $M$ satisfies the \emph{Auslander condition} if for all $i \geq 0$ and all non-zero submodules $N$ of $\mathrm{Ext}^i_A (M, A)$, 
$$j_A(N) \; \geq \; i.$$
\item[(viii)] $A$ is \emph{Auslander Gorenstein} if $\mathrm{injdim}(A) < \infty$ and all its finitely generated modules $M$ satisfy the Auslander condition. It is \emph{Auslander regular} if in addition $\mathrm{gldim}(A) < \infty$.
\end{enumerate}
\end{definition}

\medskip

\begin{remarks}\label{xxrems2.3}(i) The terminology in Definition \ref{xxdefn2.1}(iii), (iv) was introduced in the setting of a connected $\mathbb{N}$-graded $\mathbb{K}$-algebra $A$ by Levasseur in \cite[$\S$6.1]{Lev92}, motivated by Artin and Schelter, who had used in the same setting the term \emph{regular} in \cite{AS87} to describe what, following Levasseur's lead, was subsequently called a (graded) $AS$-regular algebra. In the $\mathbb{N}$-graded setting only the unique graded simple module $\mathbb{K}$ features in condition (AS2) of (iii), and in its right-hand version. The present more general form of these definitions is modelled on the usage in \cite[Definition 3.1]{WZ02}.

\medskip 

\noindent(ii) The introduction of the Auslander condition in a noncommutative noetherian setting originates with Bjork and Ekstrom \cite{Bj87, Bj89, BjEk90}.

\medskip

\noindent(iii) For an affine $\mathbb{K}$-algebra $R$ which is a commutative domain,
$$\mathrm{injdim}(R) < \infty \; \Longleftrightarrow \; R \textit{ is Auslander Gorenstein } \Longleftrightarrow\; R \textit{ is } AS\textit{-Gorenstein } $$
by \cite{Ba63}. Parallel equivalences hold when $\mathrm{gldim}(R)$ is finite. Moreover, such an algebra $R$ is Cohen Macaulay in the sense of the familiar commutative definition \cite[$\S$18.2]{Ei95} if and only if  (\ref{xxeqn2.2}) is satisfied by $R$ and its finitely generated modules.

\medskip
\noindent(iv) If $M$ is a non-zero finitely generated $A$-module then $j_A (M) \leq \mathrm{injdim}(A)$ by \cite[Remark 2.2(1)]{Lev92}.

\medskip

\noindent(v) Assume that the noetherian $\mathbb{K}$-algebra $A$ is Auslander Gorenstein with $\mathrm{injdim}(A) = d$ and let $\mathcal{M}$ denote the category of finitely generated left $A$-modules. Then Levasseur proved \cite[Proposition 4.5]{Lev92} that the map 
$$ \rho : \mathcal{M} \longrightarrow  \{0, 1, \ldots , d\} \cup \{-\infty\}  :  M \mapsto d - j_A (M)$$
is an exact partitive\footnote{A dimension function $\rho$ on $\mathcal{M}$ is \emph{exact} if $\rho (M) = \mathrm{max}\{\rho (B),\, \rho (C) \}$ for all short exact sequences $0 \longrightarrow B \longrightarrow M \longrightarrow C \longrightarrow 0$ in $\mathcal{M}$. And $\rho$ is \emph{partitive} if, whenever $M = M_0 \supseteq M_1 \supseteq \cdots M_i \supseteq M_{i+1}\supseteq \cdots $ is a descending chain in $\mathcal{M}$ with $\rho (M) = t$, then $\rho (M_i/M_{i+1}) <t$ for all $i \gg 0$.} 
dimension function. Thus the possession of the $GK$-Cohen Macaulay property by an Auslander Gorenstein algebra $A$ simply means that $\rho$ coincides with $GK$-dimension on finitely generated $A$-modules. Note that it is not at present known whether $GK$-dimension is always exact for noetherian algebras, so the known exactness of $\rho$ ensures that this is the case for an Auslander Gorenstein $GK$-Cohen Macaulay algebra. And conversely, the facts \cite[Corollary 5.4, Proposition 5.6]{KL00} that $GK$-dimension is symmetric and ideal invariant\footnote{A dimension function $\delta$ is \emph{symmetric} if its value is the same on the right or left of any $A$-bimodule which is finitely generated on both sides; and it is \emph{ideal invariant} if $\delta(M) \leq \delta(M \otimes_A I)$ for all finitely generated $A-modules$ and ideals $I$.} imply that these properties are shared by $\rho$ when $A$ is Auslander Gorenstein and $GK$-Cohen Macaulay.

\medskip

\noindent(vi)  One can also define the Krull-Cohen Macaulay condition for $A$ by replacing the Gelfand-Kirillov dimension by the Krull dimension throughout Definition \ref{xxdefn2.1}(vi), and the obvious version of (v) applies. See \cite[Chapter 6]{McCR87} for background on Krull dimension.

\end{remarks} 

\bigskip

When applied to a Hopf $\mathbb{K}$-algebra $H$ the above definitions enjoy some important simplifications and inter-relationships, which we now summarise.

\begin{lemma}\label{xxlem2.4} Let $H$ be a noetherian Hopf $\mathbb{K}$-algebra. In parts (ii) and (iii) assume that $\mathrm{injdim}(H) = d < \infty$.
\begin{enumerate}
\item[(i)] $\mathrm{prdim}({}_{\varepsilon}\mathbb{K}) = \mathrm{gldim}(H)$, and if the global dimension is finite, then $\mathrm{prdim}(V) = \mathrm{gldim}(H)$ for every non-zero finite dimensional $H$-module $V$.
\item[(ii)] If $ \mathrm{dim}_{\mathbb{K}}(\mathrm{Ext}_H^i({}_{\e}\mathbb{K}, {}_H H)) < \infty$ and $\mathrm{dim}_{\mathbb{K}}(\mathrm{Ext}_H^i(\mathbb{K}_{\e},  H_H)) < \infty $
for all $i$, $\qquad \qquad$and $\mathrm{Ext}_H^d({}_{\e}\mathbb{K}, {}_H H) \neq 0$, then $H$ is $AS$-Gorenstein.

\item[(iii)] If $H$ is $AS$-Gorenstein then
$$  \mathrm{dim}_{\mathbb{K}}(\mathrm{Ext}_H^d({}_{\e}\mathbb{K}, {}_H H)) \; = \; \mathrm{dim}_{\mathbb{K}}(\mathrm{Ext}_H^d(\mathbb{K}_{\e},  H_H)) \; = \; 1.$$

\item[(iv)] If the antipode $S$ of $H$ is bijective then the requirements of Definition \ref{xxdefn2.1}(i),(iii),(vi), (vii),(viii) only need to be imposed on the left or the right, not on both sides. 

\item[(v)] If $H$ is Auslander Gorenstein (resp. Auslander
regular) and $GK$-Cohen Macaulay, then it is $AS$–Gorenstein (resp. $AS$-regular).

\end{enumerate}
\end{lemma}

\begin{proof}\noindent(i) As already noted, the right and left global dimensions of $H$ are equal, \cite[Exercise 4.1.1]{We94}. The first equality was noted in \cite[$\S$2.4]{LL95}, a consequence of the fact that if ${}_H M$ is any $H$-module then $M\otimes_{\mathbb{K}} H$ is a free $H$-module by the Fundamental Theorem of Hopf Modules, \cite[Theorem 1.9.4]{Mo93}\footnote{Alternatively and more directly: $\mathrm{Ext}_H^i(M \otimes H, U) \cong \mathrm{Ext}_H^i(M, \mathrm{Hom}_{\mathbb{K}}(H,U))$ for all $i \geq 0$ and all $H$-modules $U$ by \cite[Proposition 1.3]{BG97}, and $\mathrm{Hom}_{\mathbb{K}}(H,U)$ is an injective $H$-module because the functor $\mathrm{Hom}_{\mathbb{K}}(H,-)$, being the right adjoint of the (exact) forgetful functor, preserves the injectivity of ${}_{\mathbb{K}}U$.} , so one can tensor a projective resolution of ${}_{\varepsilon}\mathbb{K}$ by $M$ to get a resolution of $M$. For the second equality see \cite[Proposition 1.6]{BG97}.

\medskip

\noindent(ii),(iii) These follow from \cite[Lemma 3.2]{BZ08} and its proof.
\medskip

\noindent(iv) This is an easy exercise.
\medskip

\noindent(v) This is \cite[Lemma 6.1]{BZ08}.
\end{proof}

\medskip
It has become apparent that the one-dimensional $H$-bimodules highlighted in Lemma \ref{xxlem2.4}(iii) are fundamental to the structure and behaviour of $AS$-Gorenstein Hopf algebras.  Keeping the notation and hypotheses of the lemma and following \cite[Definition 1.1]{LWZ07}, we write
\begin{equation}\label{integraleqn}  \int_H^{\ell} \, := \, \mathrm{Ext}_H^d({}_{\e}\mathbb{K}, {}_H H) \, \textit{ and } \, \int_H^r \, := \,  \mathrm{Ext}_H^d(\mathbb{K}_{\e},  H_H),
\end{equation}
and call these bimodules respectively the \emph{left integral} and \emph{right integral} of $H$. This generalises the long-standing usage when $H$ is a finite dimensional Hopf algebra, as in \cite[Definition 2.1.1]{Mo93}.

\medskip

\begin{remarks}\label{xxrems2.5}(i) The second conclusion in Lemma \ref{xxlem2.4}(i) is in general false if $\mathrm{gldim}(H) = \infty$. Consider, for example, the group algebra $\mathbb{K}D$, where $\mathbb{K}$  has characteristic 2 and $D = \langle a, b\, : \, b^2 = 1, bab = a^{-1} \rangle$, the infinite dihedral group. Then $\mathrm{gldim}(\mathbb{K}D) = \infty$, but every simple $\mathbb{K}D$-module $V$ has $\mathrm{prdim}(V) = 1$, apart from the 4 with $\mathbb{K}$-dimension 1.
\medskip

\noindent(ii) Lemma \ref{xxlem2.4}(iv) immediately suggests the 

\begin{question}\label{xxqn2.6} Is the antipode of every noetherian Hopf algebra bijective?
\end{question}

\noindent A positive answer to this question was conjectured by Skryabin \cite[page 623]{Sk06}. We'll review progress on this question in $\S$\ref{xxsubsec4.1}.

\medskip
\noindent(iii) The proof of Lemma \ref{xxlem2.4}(v) given in \cite{BZ08} would work also in the case where $H$ is Krull-Cohen Macaulay rather than $GK$-Cohen Macaulay if it could be shown in that case that the non-zero artinian right $H$-module $\mathrm{Ext}_H^d({}_{\e} \mathbb{K}, {}_H H)$ has finite $\mathbb{K}$-dimension. 

\end{remarks}

\bigskip

\section{Classes of Gorenstein and $GK$-Macaulay Hopf algebras}\label{xxsec3}

After a couple of preliminary lemmas, we record in $\S$\ref{subsecclasses} the progress which has been made to date towards settling Question \ref{xxqn1.1}. Then, in $\S$\ref{subsectPI} we focus on a particular class of examples, namely the affine noetherian Hopf $\mathbb{K}$-algebras which satisfy a PI, since the progress achieved in that setting may serve as an imperfect template for more general future results.

\subsection{Filtered-graded equipment}\label{xxsubsec3.1}

Initiated by Bjork and Ekstrom \cite{Bj87, Bj89, Ek89, BjEk90}), the analysis of relations between the possession by a filtered $\mathbb{K}$-algebra $A$ of the various homological properties listed in Definition \ref{xxdefn2.1} and by its associated graded algebra $\mathrm{gr}(A)$ appeared in a number of papers from the mid-1980s onwards - most notably, see \cite{Lev92}.  The following two lemmas build on this earlier work.

\begin{lemma}\label{xxlem3.1} (\cite[Lemma 4.4]{StZh94}) Let $A = \cup_{i \geq 0} \Gamma_i(A)$ be an $\mathbb{N}$-filtered $\mathbb{K}$-algebra with $\qquad$
$\mathrm{dim}_{\mathbb{K}}(\Gamma_0 (A))< \infty$. Suppose that $\mathrm{gr}_{\Gamma}(A)$ is noetherian, Auslander Gorenstein and graded $GK$-Cohen Macaulay (meaning that (\ref{xxeqn2.2}) is satisfied by all graded $\mathrm{gr}_{\Gamma}(A)$-modules $M$). Then $A$ is noetherian, Auslander Gorenstein and $GK$-Cohen Macaulay.
\end{lemma}

\medskip

\begin{lemma}\label{xxlem3.2} Let $A$ be a noetherian  Auslander Gorenstein $\mathbb{K}$-algebra. Let $B$ be either  $ A[X; \sigma, \delta]$ or $ A[X^{\pm 1}; \sigma]$, respectively an Ore extension or a skew Laurent extension of $A$, with $\sigma$ a $\mathbb{K}$-algebra automorphism of $A$ and $\delta$ a $\sigma$-derivation of $A$. 
\begin{enumerate}
\item[(i)] $B$ is Auslander Gorenstein, with 
$$ \mathrm{injdim}(A) \; \leq \; \mathrm{injdim}(B) \; \leq \; \mathrm{injdim}(A) + 1. $$
\item[(ii)] If $A$ is Auslander regular then so is $B$, with 
$$\mathrm{gldim}(A)  \; \leq \; \mathrm{gldim}(B) \; \leq \; \mathrm{gldim}(A) + 1.$$
\item[(iii)] If $B$ is a Hopf algebra with Auslander regular Hopf subalgebra $A$, and $A$ is semiprime, then 
$$ \mathrm{gldim}(B) \; = \; \mathrm{gldim}(A) + 1.$$
\item[(iv)] Suppose that $A = \oplus_{i \geq 0} A_i$ is a connected $\mathbb{N}$-graded $\mathbb{K}$-algebra (i.e. $A_0 = \mathbb{K}$) with $\sigma (A_i) \subseteq A_i$ for all $i$), and that $A$ is noetherian Auslander Gorenstein and $GK$-Cohen Macaulay. Then $B$ is $GK$-Cohen Macaulay.
\end{enumerate}
\end{lemma}

\begin{proof} (i) For $B = A[X; \sigma, \delta]$, the fact that $B$ is Auslander Gorenstein is \cite[Theorem 4.2]{Ek89}. For $B = A[X^{\pm 1} ; \sigma]$, note first that the subalgebra $B_0 := A[X; \sigma]$ is Auslander Gorenstein by the first case, and then, setting $\mathcal{S} := \{X^i : i \geq 0 \}$, the localisation $B = B_0\mathcal{S}^{-1}$ remains Auslander Gorenstein by \cite[Proposition 2.1]{ASZ99}. The displayed inequalities are proved in \cite[Proposition 1.9]{Y97} for $B = A[X; \sigma, \delta]$; for $B = A[X^{\pm 1}; \sigma]$, the inequalities for $B_0 := A[X; \sigma]$ are a special case of the first part, and it is easy to check that they follow from this for the localisation $B = B_0\langle X^{\pm 1}\rangle$.
\medskip

\noindent(ii) Suppose that $\mathrm{gldim}(A) < \infty$. Then the inequalities are given for the two options for $B$ respectively by parts (i) and (ii) of \cite[Theorem 7.5.3]{McCR87}.
\medskip

\noindent(iii) Let $\mathrm{gldim}(A) = n < \infty$. In view of Lemma \ref{xxlem2.4}(i) we only need to show that $\mathrm{prdim}_B({}_{\varepsilon}\mathbb{K}) = n + 1$. For both choices of $B$ this follows from \cite[Corollary 7.9.18]{McCR87}.

\medskip
\noindent(iv) This is \cite[Lemma]{LevSt93}\footnote{It is assumed in \cite{LevSt93} that $A$ is Auslander regular, but the same proof works for $A$ Auslander Gorenstein.}.
\end{proof}

\medskip

\begin{remarks}\label{xxremsgldim}(i) Presumably the analogous statement to Lemma \ref{xxlem3.2}(iii) for $A$ Auslander Gorenstein and with injective dimension replacing global dimension should also be true; but such a result does not seem to be in the literature. 

\medskip

\noindent(ii) The hypothesis in Lemma \ref{xxlem3.2}(iii) that $A$ is semiprime may well be unnecessary, both for the global dimension statement and for the putative variant for  injective dimension. In the case of global dimension, the semiprime assumption may even be redundant - see below, Theorem \ref{gldimthm}(ii) and Remark \ref{gldimthmremks}(ii). 
\end{remarks} 
\bigskip

\subsection{Classes of $AS$-Gorenstein Hopf algebras}\label{subsecclasses}
 
Notes on the proofs and attributions of the parts of the following portmanteau theorem are given after its statement. First we supply some brief details regarding definitions and notation. 

For part (v), note that a Hopf algebra is \emph{connected} if its coradical is the base field $\mathbb{K}$ - in other words, $H$ is pointed with trivial group of group-likes $G(H)$. 

The quantised function algebra $\mathcal{O}_q(G)$ of part (vi) is defined as follows. Let $\mathbb{K}$ be an arbitrary field containing a non-zero element $q$ which is \emph{not} a root of unity, let $\mathfrak{g}$ be a (finite dimensional) simple Lie algebra defined over $\mathbb{K}$, and let $U_q(\mathfrak{g})$ be its quantised enveloping algebra as in part (iii). Then define $\mathcal{O}_q(G)$ to be the Hopf subalgebra of the restricted dual $U_q(\mathfrak{g})^{\circ}$ consisting of the matrix coefficients of all the finite dimensional simple type 1 representations of $U_q (\mathfrak{g})$. For further details see \cite[page 2]{Yak14} or Joseph's book \cite{Jo95}. 

In part (ix), $H(\g)$ and $\widehat{H}(\g)$ respectively denote bosonisations of $U(\g)$  with respect to the groups $C_2$ and $C_{\infty}$; for the definitions of these Hopf algebras, see \cite[Definition 2.43]{ABS26}. 

As for the noetherian cocommutative Hopf algebras, in referring to those ``of known type'' in part (x) we mean those whose group of grouplikes is polycyclic-by-finite and whose Lie algebra of primitive elements is finite dimensional.

\begin{theorem}\label{xxthm3.3} Except where otherwise stated below, $\mathbb{K}$ is an arbitrary field.  
\begin{enumerate}
\item[(i)] Let $H$ be  affine noetherian Hopf $\mathbb{K}$-algebra satisfying a polynomial identity. Then $H$ is Auslander Gorenstein, $AS$-Gorenstein and $GK$-Cohen Macaulay.
\item[(ii)] Let $\mathfrak{g}$ be a finite dimensional Lie algebra over $\mathbb{K}$. Then the enveloping algebra $U(\mathfrak{g})$ is Auslander regular, $AS$-regular and $GK$-Cohen Macaulay, with 
$$\mathrm{gldim}(U(\mathfrak{g})) \; = \ \; \mathrm{GKdim}(U(\mathfrak{g})) \; = \; \mathrm{dim}_{\mathbb{K}}(\mathfrak{g}).$$
\item[(iii)] Let $q \in \mathbb{C}^{\ast}$ and let $\mathfrak{g}$ be a complex semisimple Lie algebra. Then the quantised enveloping algebra $U_q (\mathfrak{g})$ is Auslander regular, $GK$-Cohen Macaulay and $AS$-regular, with 
$$\mathrm{gldim}(U_q(\mathfrak{g}))\; = \; \mathrm{GKdim}(U_q(\mathfrak{g})) \; = \; \mathrm{dim}_{\mathbb{C}}(\mathfrak{g}).$$
\item[(iv)] Let $G$ be a polycyclic-by-finite group with Hirsch number\footnote{The Hirsch number of a polycyclic-by-finite group $G$ is the number of infinite cyclic subfactors in a finite series of subnormal subgroups from $\{1\}$ to $G$ with successive factors either finite or cyclic.} $h(G)$ and let $\mathbb{K}G$ be the group algebra.
\begin{itemize}
\item[(a)] $\mathbb{K}G$ is Auslander Gorenstein and $AS$-Gorenstein, with $\mathrm{injdim}(\mathbb{K}G) = h(G)$.
\item[(b)] $\mathbb{K}G$ is Auslander regular and $AS$-regular $\Longleftrightarrow$ $\mathbb{K}$ has characteristic 0, or $\mathbb{K}$ has characteristic $p > 0$ and $G$ has no element of order $p$.
\item[(c)] $\mathrm{GKdim}(\mathbb{K}G)$ is finite if and only if $G$ is nilpotent-by-finite. 
\item[(d)] Suppose that $G$ is nilpotent-by-finite and let $N$ be any nilpotent subgroup of $G$ with $|G\, : \, N| < \infty$. Let $N = Z_0 \supset Z _1 \supset \cdots \supset Z_c = \{1 \}$ be the lower central series of $N,$ with $d_i := h(Z_{i-1}/Z_{i})$ for $i = 1, \ldots , c$. Then 
\begin{equation}\label{xxeqn3.4} \mathrm{GKdim}(\mathbb{K}G) \; = \; \sum_{i = 1}^c   id_i .      
\end{equation}
\item[(e)] $\mathbb{K}G$ is $GK$-Cohen Macaulay if and only if $G$ is abelian-by-finite.
\end{itemize}

\item[(v)] Let $H$ be a noetherian Hopf $\mathbb{K}$-algebra satisfying hypothesis $\mathbf{\left(\mathcal{F}\right)}$:
\begin{align*} \mathbf{\left(\mathcal{F}\right)} \quad H& \textit{ has an ascending filtration }  H = \cup_{i \geq 0} H_i \textit{ with } H_0 = \mathbb{K}  \textit{ such that } \mathrm{gr}H \textit{ is connected}\\
&\textit{graded noetherian and every non-simple prime graded factor of } \mathrm{gr}H \\
&\textit{ contains a homogeneous normal element of positive degree.} 
\end{align*}
Then $H$ is Auslander Gorenstein, $GK$-Cohen Macaulay and $AS$-Gorenstein.

\item[(vi)]  Let $\mathcal{O}_q(G)$ be a quantised function algebra, where $q \in \mathbb{K}$ is non-zero and not a root of unity, and let $I$ be a Hopf ideal of $\mathcal{O}_q (G)$. Then $\mathcal{O}_q (G)/I$ is Auslander Gorenstein, $GK$-Cohen Macaulay and $AS$-Gorenstein. Moreover regularity holds for $\mathcal{O}_q(G)$, with 
$$\mathrm{gldim}(\mathcal{O}_q (G)) \; = \; \mathrm{GKdim}(\mathcal{O}_q (G)) \; = \; \mathrm{dim}(G).$$

\item[(vii)]  Assume that $\mathbb{K}$ is algebraically closed of characteristic 0, and let $H$ be a connected Hopf $\mathbb{K}$-algebra with $\mathrm{GKdim}(H) = \ell < \infty$. Then:
\begin{itemize}
\item[(a)] $\ell \in \mathbb{Z}_{\geq 0}$ and $H$ is a noetherian domain with $\mathrm{gldim}(H) = \ell$ and $\mathrm{Kdim}(H) \leq \ell$;
\item[(b)] $H$ is Auslander regular and $GK$-Cohen Macaulay;
\item[(c)] $H$ is $AS$-regular.
\end{itemize}

\item[(viii)]   Assume that $\mathbb{K}$ is algebraically closed of characteristic 0, and let $H$ be a Hopf $\mathbb{K}$-algebra which is connected graded as an algebra, $H = \oplus_{i \geq 0}H(i)$ with $\mathrm{dim}_{\mathbb{K}}(H(i)) < \infty$ for all $i$.  Suppose that $\mathrm{GKdim}(H) = d < \infty$. Then $d \in \mathbb{Z}_{\geq 0}$ and $H$ is a noetherian domain which is Auslander-regular, $GK$-Cohen–Macaulay and $AS$-regular with $\mathrm{gldim}(H) = d$.

\item[(ix)] Let $\g = \g_0 \oplus \g_1$ be a finite dimensional Lie superalgebra with $\mathrm{dim}_{\mathbb{K}}(\g_0) = m$.
\begin{itemize}
\item[(a)] The algebras  $U(\g)$, $H(\g)$ and $\widehat{H}(\g)$  are all Auslander-Gorenstein, AS-Gorenstein and GK-Cohen Macaulay, having injective and GK-dimensions respectively $m$, $m$ and $m + 1$.
\item[(b)] Moreover $$ \qquad \begin{aligned} \mathrm{gldim}(U(\g)) < \infty &\Longleftrightarrow \mathrm{gldim}(H(\g)) < \infty 
\Longleftrightarrow \mathrm{gldim}(\widehat{H}(\g)) < \infty \\ & \Longleftrightarrow [x,x] \neq 0 \;\; \forall\, 0 \neq x \in \g_1\\
& \Longleftrightarrow U(\g) \textit{ is a domain}.
\end{aligned}$$
When finite, the values of the global dimensions coincide with the corresponding injective dimensions.
\end{itemize}

\item[(x)] Assume that $\mathbb{K}$ is algebraically closed of characteristic 0, and let $H$ be a cocommutative noetherian Hopf algebra of known type (in the terminology explained before the theorem), so that $H = U(\mathfrak{g})\#T$, where $\mathfrak{g}$ is the finite dimensional Lie algebra of primitive elements and $T$ is the polycyclic-by-finite group of grouplikes. Then $H$ is Auslander regular and AS-regular, with 
$$\mathrm{gldim}(H) \; = \; \mathrm{dim}_{\mathbb{K}}(\mathfrak{g}) + h(T),$$
where $h(T)$ denotes the Hirsch length of $T$.
\end{enumerate}
\end{theorem} 

\begin{proof} (i) This is \cite[Theorems 0.1 and 0.2]{WZ02}.

\medskip
\noindent(ii) If $\mathbb{K}$ has positive characteristic the result is a special case of part (i), so assume that $\mathbb{K}$ has characteristic 0. With respect to the standard filtration on $U(\mathfrak{g})$, $\mathrm{gr}U(\mathfrak{g})$ is a commutative polynomial $\mathbb{K}$-algebra on $\mathrm{dim}_{\mathbb{K}}(\mathfrak{g})$ variables, so it is immediate from Lemma \ref{xxlem3.1} that $U(\mathfrak{g})$ is Auslander Gorenstein and $GK$-Cohen Macaulay. So $U(\mathfrak{g})$ is $AS$-Gorenstein by Lemma \ref{xxlem2.4}(iv). Then $\mathrm{gldim}(U(\mathfrak{g})) = \mathrm{dim}_{\mathbb{K}}(\mathfrak{g})$ by \cite[Theorem XIII.8.2]{CE56}, so the proof is complete, with the value of $\mathrm{GKdim}(U(\mathfrak{g}))$ following either from the Cohen Macaulay property, or directly from \cite[Example 6.9]{KL00}.

\medskip

\noindent(iii) If $q$ is a root of unity then the result is a special case of part (i), but was proved earlier in \cite[Proposition 2.2, Theorem 2.3]{BG97}. Assume now that $q$ is generic. Then Auslander regularity was proved in \cite[Proposition 2.2]{BG97}. But the fact that $U_q (\mathfrak{g})$ is $GK$-Cohen Macaulay needs more delicate filtered graded arguments, which are provided in \cite{GTL04}. From this the $AS$-regularity of $U_q(\mathfrak{g})$ follows by Lemma \ref{xxlem2.4}(iv). That $\mathrm{gldim}(U_q(\mathfrak{g})) = \mathrm{dim}_{\mathbb{K}}(\mathfrak{g})$ is then given by \cite[Proposition 3.2.1]{Ch04} together with the $AS$-regularity.  Finally, the $GK$-Cohen Macaulay property forces $\mathrm{GKdim}(U_q(\mathfrak{g})) = \mathrm{gldim}(U(\mathfrak{g}))$.

\medskip

\noindent(iv)(a),(b) This is \cite[Theorem 6.7]{BZ08}.
\medskip

\noindent(c) This is Gromov's theorem, \cite{Gr81}. See also \cite[Chapter 11]{KL00}.
\medskip

\noindent(d)  This is due to Bass \cite{Ba72} and Guivarc'h \cite{Gu73}; see also \cite[Theorem 11.14]{KL00}.
\medskip

\noindent(e)$\Longrightarrow$ Suppose that $\mathbb{K}G$ is $GK$-Cohen Macaulay. Then $\mathrm{GKdim}(\mathbb{K}G) < \infty$, so $G$ is nilpotent-by-finite by (d). Since $\mathbb{K}G$ is $AS$-Gorenstein with $\mathrm{injdim}(\mathbb{K}G) = h(G)$ by (a), condition (AS2) of Definition \ref{xxdefn2.1}(iii) applied to ${}_{\varepsilon}\mathbb{K}$ yields
\begin{equation}\label{xxeqn3.5}    j({}_{\varepsilon}\mathbb{K}) \; = \; h(G). 
\end{equation}
On the other hand the $GK$-Cohen Macaulay property implies that
\begin{equation}\label{xxeqn3.6}  j({}_{\varepsilon}\mathbb{K}) \; = \; \mathrm{GKdim}(\mathbb{K}G).
\end{equation}
From (\ref{xxeqn3.5}) and (\ref{xxeqn3.6}) we deduce that $\mathrm{GKdim}(\mathbb{K}G) = h(G)$. From this together with (\ref{xxeqn3.4}) it follows that
$$ d \; = \; \sum_{i=1}^c d_i \; = \; \sum_{i=1}^c id_i, $$
which is only possible if $c = 1$, so that $G$ is abelian-by-finite, as claimed.

\noindent$\Longleftarrow$  Suppose that $G$ is abelian-by-finite, and recall that $G$ must be finitely generated, since by hypothesis it is polycyclic-by-finite. Therefore $\mathbb{K}G$ is affine noetherian and satisfies a polynomial identity by \cite[Theorem 5.2.14]{Pa77}, so $\mathbb{K}G$ is $GK$-Cohen Macaulay by part (i).
\medskip

\noindent(v) This is (part of) \cite[Theorem 0.4]{LWZ09}, proved using basic properties of rigid dualising complexes.
\medskip

\noindent(vi) This is \cite[Theorem 9.4 ]{Yak14}. The displayed equalities follow from the known value of the $GK$-dimension and the $GK$-Cohen Macaulay property.
\medskip

\noindent(vii) Parts (a) and (b) are results of Zhuang \cite[Corollary 6.10]{Zhu13}. Part (c) follows from (a), (b) and Lemma \ref{xxlem2.4}(iv).

\medskip
\noindent(viii) First, $d$ is a non-negative integer by \cite[Theorem 2.4]{BGZ17}. The remaining claims are proved in \cite[Theorem 2.2]{BGZ17}.

\medskip

\noindent(ix) The key points here, namely that $\mathrm{gldim}(U(\g))$ is finite and that $U(\g)$ is a domain when $[x,x] \neq 0$ for all $0 \neq x \in \g_1$,  are due to \cite{Bo84}; see also \cite{AL85}, and \cite[Chapter 17]{Mu12} for a detailed account. The consequences for $H(\mathfrak{g})$ and $\widehat{H}(\mathfrak{g})$ are deduced in \cite[Proposition 4.36]{ABS26}.

\medskip

\noindent(x) The structure of $H$ is given by the Cartier-Kostant-Gabriel theorem, \cite[Corollary 5.6.4, Theorem 5.6.5]{Mo93}.  Note first that both $U(\mathfrak{g})$ and $\mathbb{K}T$ are $AS$-regular, respectively of global dimensions $\mathrm{dim}_{\mathbb{K}}(\mathfrak{g})$ and $h(T)$, by parts (ii) and (iv). Hence the smash product $H$ is $AS$-regular with $\mathrm{gldim}(H) = \mathrm{dim}_{\mathbb{K}}(\mathfrak{g}) + h(T)$, by \cite[Proposition 7.3.1]{LM19}.
\end{proof}

\medskip

\begin{remarks}\label{xxrem3.7} (i) The hypothesis in part (i) that $H$ is affine may be redundant, since it is currently an open question whether every noetherian Hopf algebra is affine, even for PI Hopf algebras. In any case the hypothesis that $H$ is affine can be weakened to the assumption that
 $$(\ast) \qquad \textit{  every simple } H\textit{-module has finite } \mathbb{K}\textit{-dimension},$$
as was proved in \cite[Theorem 3.5]{WZ02}. Every affine PI $K$-algebra possesses property $(\ast)$ by \cite[Theorem 13.10.3(i)]{McCR87}. Some special cases of part (i) where $H$ is a finite module over its centre  were proved earlier; \cite[Theorem 1.14]{BG97}, \cite[Proposition 2.3]{Br97}.  But we note that a multitude of affine noetherian Hopf algebras have now been found which satisfy a PI but are \emph{not} finite over their centres - see \cite{ABS26}.

\medskip
\noindent(ii) Part (ii) of the theorem is a special case of part (vii), but of course was proved much earlier. In \cite[Theorem 5.6]{BGZ17} a connected Hopf $\mathbb{C}$-algebra of $GK$-dimension 5 is exhibited which is not isomorphic as an algebra to an enveloping algebra of a Lie algebra. Thus part (vii) strictly generalises part (ii). It should be noted that in \cite{HW25} a large family of connected noetherian Hopf algebras have been discovered which are not iterated Hopf Ore extensions of the enveloping algebras of their Lie algebras of primitive elements.

\medskip

\noindent(iii) Yakimov's result (vi) for $\mathcal{O}_q (G)$ was the culmination of a series of earlier special cases. Thus Lu, Wu and Zhang had shown \cite[Theorem 0.5]{LWZ09} that, for all $G$, $\mathcal{O}_q (G)$ satisfies $\mathbf{(\mathcal{F})}$ when $\mathbb{K} = \mathbb{C}(q)$, and earlier special cases had been obtained in \cite{LevSt93, BG97, GZ07}, with Yakimov's proof strategy building on that of \cite{GZ07}.\footnote{{\color{red} It's not clear to me that - for arbitrary $q$ - every detail here is covered in the literature. For example, for $G \neq GL(n,\mathbb{K})$ or $SLK(n.\mathbb{K})$, a proof of the exact value of $\mathrm{GKdim}(\mathcal{O}_q(G))$, needed for the displayed formula in (vi), seems hard to locate. And similar concerns apply to the finiteness of $\mathrm{gldim}(\mathcal{O}_q(G))$ for these groups $G$ - the authors of \cite{BG02} write in rather cavalier style on page 218 that "the general case is similar" [to the case $q$ transcendental].}}

\medskip

\noindent(iii) Regarding part (vii), it is not known whether every noetherian connected Hopf algebra in characteristic 0 has finite $GK$-dimension. Notice that the existence of an infinite dimensional Lie algebra with noetherian enveloping algebra would provide a counterexample, in view of \cite[Lemma 6.5]{KL00}. 

\medskip
\noindent (iv) It is natural to consider also the Krull-Cohen Macaulay condition. Of course by Theorem \ref{xxthm3.3}(i) this is satisfied by affine noetherian Hopf algebras $H$ satisfying a PI, since the Krull and GK-dimensions coincide in that setting by \cite[Corollary 10.16]{KL00}. But perhaps  Krull-Cohen Macaulay \emph{always} fails when the affine noetherian Hopf algebra $H$ is not PI? For example, let $H = U(\mathfrak{g})$ for a finite dimensional complex Lie algebra $\mathfrak{g}$. Suppose first that $\mathfrak{g}$ is semisimple, and suppose that $U(\mathfrak{g})$ is Krull-Cohen Macaulay, so providing the third equality in (\ref{xxeqn3.8}) below. By \cite{Le02}, $\mathrm{Kdim}(U(\mathfrak{g})) = \mathrm{dim}_{\mathbb{C}}(\mathfrak{b})$, where $\mathfrak{b}$ is a Borel subalgebra of $\mathfrak{g}$, yielding the fourth equality.  The first equality is Theorem \ref{xxthm3.3}(ii) and the second follows from Lemma \ref{xxlem2.4}(v). This yields the contradiction
\begin{equation}\label{xxeqn3.8}\mathrm{dim}_{\mathbb{C}}(\mathfrak{g}) \; = \; \mathrm{gldim}(U(\mathfrak{g})) \; = \; \mathrm{Kdim}({}_{\varepsilon}\mathbb{C}) +  j({}_{\varepsilon}\mathbb{C})  \; = \; \mathrm{Kdim}(U(\mathfrak{g})) \; = \; \mathrm{dim}_{\mathbb{C}}(\mathfrak{b}) . 
\end{equation}
On the other hand let $\mathfrak{g}$ be the 2-dimensional complex solvable non-abelian Lie algebra, spanned by $x$ and $y$ with $[y,x] = x$. Take $V = U(\mathfrak{g})/U(\mathfrak{g})(x - 1)$, easily shown to be a simple module with projective dimension 1. Since $\mathfrak{g}$ is soluble, $\mathrm{Kdim}(U(\mathfrak{g})) = \mathrm{dim}_{\mathbb{C}}(\mathfrak{g})$ by \cite[Theorem 6.6.2]{McCR87}, so the Krull version of (\ref{xxeqn2.2}) yields the contradiction 
$$    1 \; = \; 1 + 0 \; = \; j(V) + \mathrm{Kdim}(V) \; = \; \mathrm{Kdim}(U(\mathfrak{g})) \; = \; 2.$$
A similar calculation shows that the enveloping algebra of the 3-dimensional Heisenberg Lie algebra over $\mathbb{C}$ is not Krull-Cohen Macaulay.

\medskip

\noindent(v) Le Meur's smash product result used in the proof of part (x) is stated in terms of the skew Calabi-Yau property, but this is equivalent to $AS$-regularity for noetherian Hopf $\mathbb{K}$-algebras thanks to \cite[Lemma 1.3]{RRZ14}. We'll review this equivalence in detail in $\S$\ref{xxsubsec4.2}.

\medskip

\noindent (vi) In \cite{RWZ21} it is shown that a weak Hopf algebra $H$ which is a finitely generated module over its affine centre is Auslander-Gorenstein, $GK$-Macaulay and $AS$-Gorenstein.  Then, in \cite{RWZ21} and in its sequel \cite{RWZ26}, it is proved that some of the structural properties described in Theorem \ref{xxthm3.9} below, for affine noetherian Hopf $\mathbb{K}$-algebras satisfying a polynomial identity, are also possessed (in a suitably generalised form) by weak Hopf algebras which are finite modules over their centres. Clearly one should cosider whether some of these conclusions may remain true for other classes of noetherian weak Hopf algebras. 

\end{remarks}

\bigskip

\subsection{Noetherian Hopf algebras satisfying a polynomial identity}\label{subsectPI} In this subsection we give further details about the PI case covered by Theorem \ref{xxthm3.3}(i).  It is perhaps worth noting from the signposts given below to the proof of Theorem \ref{xxthm3.9} that the structure theory for an affine noetherian PI Hopf algebra of finite \emph{global} dimension was known before the appearance of the seminal paper \cite{WZ02}, which extended our knowledge to the non-regular case. 

Recall also that a Frobenius algebra of finite global dimension is semisimple, so that part (iii) of the theorem below is a generalisation of the famous result of Larson and Radford \cite{LR88} that, in characteristic 0, finite dimensional involutory Hopf algebras are semisimple.

\begin{theorem}\label{xxthm3.9} Let $\mathbb{K}$ be an arbitrary field, and let $H$ be an affine noetherian Hopf $\mathbb{K}$-algebra satisfying a polynomial identity. Denote the centre of $H$ by $Z$.
\begin{enumerate}
\item[(i)] Suppose that $\mathrm{gldim}(H)$ is finite, say $\mathrm{gldim}(H) = d$. 
\begin{itemize}
\item[$(a)$] $H$ is a finite direct sum of prime algebras, $H = \oplus_{i=1}^t H/P_i$. 
\end{itemize}
For $i = 1, \ldots, n$, denote by $H_i$ the (uniquely determined) summand of $H$ such that $H = H_i \oplus P_i$. Set $Z_i := Z \cap H_i$, so $Z_i = Z(H_i)$ and $Z = \oplus_{i = 1}^n Z_i$.
\begin{itemize}
\item[$(b)$] For each $i = 1, \ldots , n$, $H_i$ is a finitely generated Cohen Macaulay module over the integrally closed affine domain $Z_i$ and so $H$ is a finitely generated $Z$-module.
\item[$(c)$] For each $i = 1, \ldots , n$, $H_i$ is an Auslander regular, $AS$-regular and $GK$-Cohen Macaulay prime affine noetherian PI $\mathbb{K}$-algebra with
$$ \mathrm{gldim}(H_i) \; = \;  \mathrm{GKdim}(H_i) \; = \; d. $$
The same conclusions therefore apply to $H$.
\item[(d)] Every maximal ideal of $H$ or of  $H_i$ has height $d$.
\item[(e)] If $C$ is a commutative regular subalgebra of $H$ [resp. of $H_i$] over which $H$ [resp $H_i$] is a finitely generated module, then $H$ [resp.$H_i$] is a projective $C$-module.
\end{itemize}
\item[(ii)] Suppose that $\mathrm{gldim}(H) = \infty$. 
\begin{itemize}
\item[(a)] Even if $\mathbb{K}$ has characteristic 0, $H$ may not be semiprime. Even if $H$ is prime, it may not be a finite $Z$-module.
\item[(b)] Consider the three statements:
\begin{itemize}
\item[$(A)$] $Z$ is noetherian;
\item[$(B)$] $Z$ is an affine $\mathbb{K}$-algebra;
\item[$(C)$] $H$ is a finitely generated $Z$-module.
\end{itemize}
Then $(C)\Longrightarrow (B) \Longrightarrow (A)$. If $H$ is semiprime then $(A), \, (B)$ and $(C)$ are equivalent. 
\end{itemize}

\item[(iii)] Assume that $\mathbb{K}$ has charactersitic 0, and suppose that $H$ is involutory (that is, $S^2 = \mathrm{id}_H$). Then the following are equivalent:
\begin{itemize}
\item[(A)] $\mathrm{gldim}(H) < \infty$;
\item[(B)]  $H$  is a finitely generated $ Z$-module or $H$ is semiprime;
\item[(C)] $H$  is a finitely generated $ Z$-module and $H$ is semiprime.
\end{itemize}
\end{enumerate}
\end{theorem}
\begin{proof}(i) For a noetherian PI Hopf algebra $H$ with $\mathrm{gldim}(H) < \infty$, Auslander regularity and the $GK$-Cohen Macaulay property were known before the more general results of Theorem \ref{xxthm3.3}(i). Namely these conclusions were assured  by \cite[Corollary 1.6, page 246]{BG97}, which is an immediate consequence of the homogeneity given by Lemma \ref{xxlem2.4}(i) combined with  \cite[Proposition 5.2 and Theorem 5.4(i)]{StZh94}. The $AS$-regularity of $H$ then follows from Lemma \ref{xxlem2.4}(v). 
\medskip

\noindent(a) Given the above, $H$ is a finite direct sum of prime algebras $H_i$ by \cite[Theorem 5.4(i)]{StZh94}. 
\medskip

\noindent(b), (c) The prime summands $H_i$ all have the same $GK$-dimension as $H$ itself, namely $d$,  thanks to the $GK$-Cohen Macaulay property coupled with the fact that the $H_i$ are projective $H$-modules. It is an easy exercise to then check that the homological properties  of $H$ are inherited by the summands $H_i$. Since $H$ is assumed to be affine, so are its images $H_i$, so the finite generation of the $Z_i$-module $H_i$ is given by  \cite[Theorem 5.4(iii)]{StZh94}. The Artin-Tate Lemma \cite[Lemma 13.9.10]{McCR87} then yields the fact that $Z_i$ is affine for all $i$. By \cite[Theorem 5.4(iii)]{StZh94} $H_i$ is equal to its own trace ring and hence $Z_i$ is integrally closed. Finally, that $H_i$ is a Cohen Macaulay $Z_i$-module is a consequence of the GK-Cohen Macaulay property, by \cite[Theorem 4.8]{BrMcL17}.
\medskip

\noindent(d) For every $i$, since $H_i$ is $\mathbb{K}$-affine prime PI with $\mathrm{GKdim}(H_i) = d$, all its maximal ideals have height $d$ by Schelter's theorem \cite[Theorem 13.10.12]{McCR87}. Since distinct minimal primes of $H$ are comaximal, each maximal ideal $M$ of $H$ contains a unique minimal prime, so the height of $M$ is also $d$.
\medskip

\noindent(e) Since $H_i$ is a Cohen Macaulay $Z_i$-module and all its maximal ideals have the same height by (d), this follows from \cite[Theorem 3.7]{BrMcL17}. The same argument applies to $H$.

\medskip

\noindent(ii)(a) The first claim is clear - the 4-dimensional Sweedler $\mathbb{C}$-algebra is the smallest such example. For the second, see the example discussed in \cite[Theorem 5.3]{ABS26}.

\noindent(b) $(C)\Longrightarrow (B)$ is the Artin-Tate Lemma \cite[Lemma 13.9.10(ii)]{McCR87}, and $(B)\Longrightarrow (A)$ is a consequence of Hilbert's Basis Theorem. When $H$ is semiprime $(A)\Longrightarrow (C)$ is given by \cite[Corollary 13.6.14]{McCR87}.

\medskip

\noindent(iii) $(A) \Longrightarrow (C)$ is a special case of (i)(a) and (i)(b), and $(C)\Longrightarrow (B)$ is trivial. Finally, $(B)\Longrightarrow (A)$ follows from  \cite[Theorems 0.1, 0.2]{WZ02a}.
\end{proof}

\medskip.

Note that part (iii) of the theorem shows that, when $H$ is involutory, the converse to part (i)(b) is true. This is certainly not the case in general, but it suggests the following question, first stated in 2002: 

\begin{question}\label{nonfgregular}(Wu, Zhang, \cite[Question 0.3]{WZ02a}) If $\mathbb{K}$ has characteristic 0 and $H$ is an affine or noetherian involutory Hopf $\mathbb{K}$-algebra, is $H$ regular?
\end{question} 


\bigskip
\subsection{Hopf algebras of low $GK$-dimension}\label{subsectlow} Throughout this subsection, assume that
$$ (\mathcal{H}) \qquad \mathbb{K} \textit{ is an algebraically closed field of characteristic } 0. $$
Some, but not all, of the results here are valid without assuming $(\mathcal{H})$. 

Since finite dimensional Hopf algebras are Frobenius algebras \cite{LS69}, \cite[Theorem 12.5]{Lo18} of $GK$-dimension 0, they are trivially Auslander Gorenstein, $AS$-Gorenstein and $GK$-Cohen Macaulay, so it follows that  they are Auslander-regular if and only if they are semisimple. 
\medskip

\noindent{\bf $GK$-dimension one.} Let's now consider an affine noetherian Hopf $\mathbb{K}$-algebra $H$ of $GK$-dimension 1. By \cite{SSW85} every affine algebra of $GK$-dimension 1 satisfies a polynomial identity. Hence,  by Theorem \ref{xxthm3.3}(i), $H$ is Auslander Gorenstein, $AS$-Gorenstein and $GK$-Cohen Macaulay, and so in particular $\mathrm{injdim}(H) = 1$. If $H$ is in addition semiprime then it is a finite module over its centre by \cite{SSW85} (and the noetherian hypothesis can be omitted, as it follows here from $H$ being affine). Importantly, the semiprime setting includes the case where $H$ is $AS$-regular in view of Theorem \ref{xxthm3.9}(i)(a),(b).

The prime affine noetherian $AS$-regular Hopf $\mathbb{K}$-algebras of $GK$-dimension 1 have in fact been classified, by the combined efforts of \cite{LWZ07}, \cite{BZ10} and \cite{WLD16}; for a summary of this classification, see \cite[$\S$1.3]{BZ21}. 

However, there also exist prime affine noetherian Hopf $\mathbb{K}$-algebras of $GK$-dimension 1 which are \emph{not} regular - a first such example was given as \cite[Example 7.3]{BZ10}, and significant further progress towards a classification of the non-regular algebras was made in \cite{Li20}. Namely, Liu in \cite{Li20} introduces two further hypotheses which he imposes on a prime affine noetherian Hopf $\mathbb{K}$-algebra. Both of these hypotheses are implied by regularity, but not conversely. He then classifies the prime affine noetherian Hopf $\mathbb{K}$-algebras which satisfy these two additional conditions, finding in the process a number of new infinite families of examples, including non-pointed ones. At the time of writing it is not known whether yet more examples remain to be discovered. 

Returning to the regular case, if we omit the hypothesis that $H$ is prime then much less is currently known. Thus, let's consider a regular affine noetherian but \emph{not} prime Hopf $\mathbb{K}$-algebra $H$ with $\mathrm{GKdim}(H) = 1$. Then Theorem \ref{xxthm3.9}(i) implies that $H$ is hereditary, (that is $\mathrm{gldim}(H) = 1$), and that $H$ is semiprime, even a finite  direct sum of prime algebras,
\begin{equation}\label{Hsplit} H \; = \; \bigoplus_{i=0}^n H_i, \qquad H_i \; = \; H/P_i,
\end{equation}
where $P_0, \ldots , P_n$ are the minimal prime ideals of $H$. Each $H_i$ is an affine noetherian hereditary prime algebra of $GK$-dimension one, and is a finite module over its centre $Z_i$, which is an affine Dedekind domain. For $i = 0, \ldots , n$ we can denote by $e_i \in H_i$ the primitive central idempotents of $H$, so that $P_i = \sum_{j \neq i}He_j$ for all $i$. By considering the images of the $e_i$ in ${}_{\varepsilon}\mathbb{K} = H/\mathrm{ker}\varepsilon$ we see that there is precisely one of the $e_i$ not mapped to 0. Let's label this idempotent $e_0$; equivalently,
\begin{equation}\label{soloP}  P_0 \textit{ is the unique minimal prime of } H \textit { in }\mathrm{ker}\varepsilon.
\end{equation}
We now have the following result, parts (i) and (ii) of which are \cite[Theorem 6.5]{LWZ07}.

\begin{theorem}\label{GKoneregularthm} Let $H$ be a regular affine noetherian Hopf $\mathbb{K}$-algebra with $\mathrm{GKdim}(H) = 1$, and retain the notation introduced in the previous paragraph.
\begin{enumerate}
\item[(i)] $P_0$ is a Hopf ideal, so that $H_0$ is a Hopf algebra listed in the classification of prime regular affine noetherian Hopf $\mathbb{K}$-algebras of $GK$-dimension one.
\item[(ii)] The coinvariant subalgebra $H^{\mathrm{co}(H/P_0)}$ contains all finite dimensional normal Hopf subalgebras of $H$.
\item[(iii)] $H$ is a left and right flat $H^{\mathrm{co}(H/P_0)}$-module.
\end{enumerate}
Suppose henceforth that $\mathrm{dim}_{\mathbb{K}} (H^{\mathrm{co}(H/P_0)}) < \infty$.
\begin{enumerate}
\item[(iv)]  $H^{\mathrm{co}(H/P_0)}$ is a Frobenius algebra, and $H$ is a free left and right $H^{\mathrm{co}(H/P_0)}$-module.
\item[(v)] The trivial module ${}_{\varepsilon} \mathbb{K}$ is a (left and right) projective $H^{\mathrm{co}(H/P_0)}$-module.
\item[(vi)] If $H^{\mathrm{co}(H/P_0)}$ is a Hopf algebra, then it is semisimple artinian.
\end{enumerate}
\end{theorem}

\begin{proof}(iii) Because $H$ is a finite module over its semiprime affine noetherian centre, it is easy to see that $H$ is residually finite dimensional. So (iii) is given by \cite[Theorem 0.1]{Sk21}.
\medskip

\noindent (iv) If $H^{\mathrm{co}(H/P_0)}$ is finite dimensional, then both claims in (iv) are given by \cite[Theorem 6.1]{Sk07}, noting that the hypotheses of that result are satisfied since $H$ is weakly finite because it is noetherian. 
\medskip

\noindent (v) The same proof applies on either side. Since $H$ is regular by hypothesis, $\mathrm{prdim}_H ({}_{\varepsilon} \mathbb{K}) = 1 = \mathrm{gldim} (H)$. Thanks to the second claim in (iv) we can restrict a projective $H$-resolution of ${}_{\varepsilon}\mathbb{K}$ to obtain an $H^{\mathrm{co}(H/P_0)}$-resolution, whence 
$$\mathrm{prdim}_{H^{\mathrm{co}(H/P_0)}}({}_{\varepsilon}\mathbb{K}) \; \leq \; 1.$$
But $H^{\mathrm{co}(H/P_0)}$ is a Frobenius algebra by the first part of (iv), so projective $H^{\mathrm{co}(H/P_0)}$-modules are injective. Hence the projective resolution of ${}_{\varepsilon} \mathbb{K}$ splits, forcing ${}_{\varepsilon}\mathbb{K}$ to be projective.

\medskip
\noindent (vi) This follows from (v) and Lemma \ref{xxlem2.4}(i).
\end{proof}

\begin{remarks}\label{GKoneregularrmks} (i) The proofs in \cite{LWZ07} of Theorem \ref{GKoneregularthm}(i),(ii)  make ingenious use of the left integral $\int^{\ell}_H$ and of the theory of links and cliques of prime ideals, \cite[Chapters 12, 14]{GW04}. It seems reasonable to hope that the theorem and its proof might be signposts towards a future better understanding of regular affine noetherian PI Hopf $\mathbb{K}$-algebras of arbitrary $GK$-dimension.

\medskip
\noindent (ii) As pointed out in \cite{LWZ07}, there are obvious parallels between Theorem \ref{GKoneregularthm} and the elementary theory of affine algebraic groups over $\mathbb{K}$. Thus if $H = \mathcal{O}(G)$ is the coordinate ring of an affine algebraic $\mathbb{K}$-group $G$ then the connected component of $G_{conn}$ of $G$ is a normal subgroup of finite index in $G$ with $\mathcal{O}(G_{conn})$ a domain, and there is an exact sequence of commutative Hopf algebras
$$ 0 \longrightarrow (\mathbb{K}(G/G_{conn}))^{\ast} \longrightarrow \mathcal{O}(G) \longrightarrow \mathcal{O}(G_{conn}) \longrightarrow 0. $$

There are similar parallels with (not necessarily PI) noetherian group algebras: let G be a group and $\mathbb{K}$ a field such that $\mathbb{K}G$ is noetherian. Then there is a unique largest finite normal subgroup of $G$, denoted $\Delta^+ (G)$, so that $G/\Delta^+ (G)$ has no non-trivial finite normal subgroups. Hence $\mathbb{K}(G/\Delta^+(G))$ is noetherian and prime, by Connel's theorem, \cite[Theorem 4.2.10]{Pa77}. Let $I$ denote the augmentation ideal of $\mathbb{K}\Delta ^+(G)$. Then $\mathbb{K}G/I\mathbb{K}G \cong \mathbb{K}(G/\Delta^+ (G))$, so that $I\mathbb{K}G$ is a minimal prime ideal and a Hopf ideal of $\mathbb{K}G$, and there is an exact sequence of cocommutative Hopf algebras
$$ 0 \longrightarrow \mathbb{K}\Delta^+(G) \longrightarrow \mathbb{K}G \longrightarrow \mathbb{K}(G/\Delta^+(G)) \longrightarrow 0. $$
I presume that with a little care and the application of the Cartier-Kostant-Gabriel structure theorem, a generalisation of this observation could be obtained for \emph{all} cocommutative noetherian Hopf $\mathbb{K}$-algebras, (recalling that in $\S$\ref{subsectlow} $\mathbb{K}$ has characteristic 0). 
\medskip

\noindent (iii) In the light of Theorem \ref{GKoneregularthm} and Remarks \ref{GKoneregularrmks}(i),(ii), we list here some questions specific to this setting, many of them already suggested in \cite[Remark 6.6]{LWZ07}. It seems to me very surprising that Question \ref{GKoneregqns}(ii) appears still to be unanswered, almost 20 years after the publication of \cite{LWZ07}.

\begin{questions}\label{GKoneregqns} Retain the notation of Theorem \ref{GKoneregularthm}
\begin{enumerate}
\item[(i)] Is  $H^{\mathrm{co}(H/P_0)}$ a Hopf subalgebra of $H$? Is it normal?
\item[(ii)] Is $\mathrm{dim}_{\mathbb{K}}(H^{\mathrm{co}(H/P_0)}) < \infty$?
\item[(iii)] Find interesting examples! In particular, are there examples of Hopf algebras $H$ fitting the template of Theorem \ref{GKoneregularthm} which are not simply group algebras, or tensor products of a finite dimensional semisimple Hopf algebra and a prime regular affine noetherian Hopf algebra of $GK$-dimension one?
\item[(iv)] To what extent is Theorem \ref{GKoneregularthm} a special case of results applying to all regular affine noetherian PI Hopf $\mathbb{K}$-algebras?
\end{enumerate}
\end{questions}
\end{remarks}

\bigskip
\noindent{\bf $GK$-dimension two.} The prime affine noetherian Hopf $\mathbb{K}$-algebras of $GK$-dimension 2 are currently far from being classified.  The first significant progress was made by Goodearl and Zhang \cite{GZ10} who classified noetherian Hopf $\mathbb{K}$-algebras $H$ satisfying the following hypotheses:
\begin{equation}\label{xxeqnGK2}  H \textit{ is a domain, } \mathrm{GKdim}(H) = 2, \; \; \mathrm{Ext}^1_{H}(\mathbb{K},\mathbb{K}) \neq 0. 
\end{equation}
In more geometric, more intuitive but less precise language, if $H$ satisfies $(\ref{xxeqnGK2})$ then it is a connected 2-dimensional quantum group with a classical subgroup of dimension at least 1. Recalling that in this subsection $\mathbb{K}$ is algebraically closed of characteristic 0, the main result of \cite{GZ10} can be summarised as follows:

\begin{theorem} \label{GKtwothm}{\rm(Goodearl-Zhang, \cite[Theorem 0.1]{GZ10})} Let $H$ be a Hopf $\mathbb{K}$-algebra satisfying $(\ref{xxeqnGK2})$. Then
\begin{enumerate}
\item[(i)] $H$ is noetherian if and only if it is affine.
\item[(ii)] If $H$ is noetherian it is isomorphic to one of the following:
\begin{itemize}
\item[(a)] the group algebras $\mathbb{K}(C_{\infty} \times C_{\infty})$ and $\mathbb{K}(C_{\infty}\rtimes C_{\infty})$;
\item[(b)] the enveloping algebras of the 2 Lie $\mathbb{K}$-algebras of dimension 2;
\item[(c)] a member of one of 3 infinite families:
\begin{itemize}
\item[$(\bullet)$] $A = A(n,q) := \mathbb{K}\langle x^{\pm 1},y \, : \, xy = qyx \rangle $, $n \in \mathbb{Z}, \, n \geq 1$, $q \in \mathbb{K} \setminus \{0\}$, $x$ group-like and $y$ $(x^n,1)$-skew primitive;
\item[$(\bullet)$] $B = B(n, p_0, \ldots , p_s, q)$, $n, p_0, \ldots , p_s$ positive integers with $s \geq 2$, and $1< p_1 < \cdots < p_s$, with $p_0 | n$ and $p_0, p_1, \ldots , p_s$ pairwise relatively prime, $q \in \mathbb{K}$ a primitive $\ell$th root of unity where $\ell = (n/p_0)p_1 \cdots p_s$;
\item[$(\bullet)$] $C = C(n) : = \mathbb{K}[y^{\pm 1}][x; (y^n - y)\partial/\partial y]$, with $y$ group-like and $x$ $(1, y^{n-1})$-skew primitive, $n \in \mathbb{Z},\, n \geq 2$. 
\end{itemize}
\end{itemize}
\end{enumerate}
\end{theorem}

The definition of $B(n, p_0, p_1, \ldots , p_s, q)$ is given at \cite[Construction 1.2]{GZ10} - it is a subalgebra $A[x^{\pm 1}; \sigma]$ of $\mathbb{K}[y][x^{\pm 1}; \sigma]$, where $\sigma \in \mathrm{Aut}(\mathbb{K}[y])$ is defined by $\sigma (y) = qy$.  The algebras listed in Theorem \ref{GKtwothm} are pairwise non-isomorphic, except that $A(n,q) \cong A(-n, q^{-1})$. One can check that these Hopf algebras satisfy the following properties:

\begin{corollary}\label{GKtwocor} Retain the hypotheses of Theorem \ref{GKtwothm}, with $H$ noetherian. Then
\begin{enumerate}
\item[(i)] $H$ is pointed;
\item[(ii)] $H$ is Auslander-Gorenstein, $AS$-Gorenstein and $GK$-Cohen Macaulay with $\mathrm{injdim}(H) = 2$;
\item[(iii)] $H$ is $AS$-regular and Auslander-regular if and only if $H$ is a member of class (a), (b), (c)(A) or (c)(C). Then $\mathrm{gldim}(H) = 2$. 
\item[(iv)] $H$ is PI if and only if $H$ is in (a), or is $U(\mathfrak{g})$ for $\mathfrak{g}$ the abelian Lie algebra in (b), or is $A(n,q)$ with $q$ a root of unity, or is from case (c)$B$.
\end{enumerate}
\end{corollary}
\begin{proof} (i) The algebras listed are all generated by their group-like and skew-primitive elements, so they are pointed by \cite[Lemma 5.5.1]{Mo93}.
\medskip

\noindent(ii) It is shown in \cite[Proposition 0.2]{GZ10} that all the listed algebras are Auslander-Gorenstein  and $GK$-Cohen Macaulay; for (a) and (c)(B) this is thanks to the PI result, Theorem \ref{xxthm3.3}(i); for (b) it is of course a special case of Theorem \ref{xxthm3.3}(ii); for cases (A) and (C) of (c) this follows by standard arguments involving skew polynomial algebras, see \cite[proof of Prop. 0.2(b), page 3166]{GZ10}. In all cases the $AS$-Gorenstein property then follows by Lemma \ref{xxlem2.4}(v).

\medskip

\noindent (iii) That the algebras in classes (a) and (b) have global dimension 2 is of course well known (and these are special cases of Theorem \ref{xxthm3.3}). The global dimension of the three classes in (c) is discussed in \cite[Proofs of Prop. 1.6 and Prop. 0.2]{GZ10}. (In fact it is easy to see that the algebras $B$ have infinite global dimension, since they are skew Laurent polynomial algebras with coefficient algebra the commutative singular affine domain $A := \mathbb{K}\langle y^{m_1}, \ldots , y^{m_s}\rangle$, where $m_i := m/p_i$, where $m = \Pi_{i=1}^s p_i$.)

\medskip
\noindent (iv) The group algebras in (a) are easily seen to be finite modules over their centres. For (b) simply note that the enveloping algebra $U$  of the 2-dimensional non-abelian Lie algebra has $Z(U) = \mathbb{K}$. That the algebras from case (c) satisfying the conditions stated in (iv) are PI (and in fact finite over their centres) is confirmed in the course of the calculations in \cite[$\S$1]{GZ10}. 
\end{proof}

\medskip

\begin{remark}\label{gldimrem} It had been hoped - see \cite[Question K]{Br07} - that every affine noetherian Hopf $\mathbb{K}$-algebra domain might have finite global dimension. But the algebras $B(n, p_0, p_1, \ldots , p_s, q)$ put paid to that, since they are affine noetherian Hopf domains which nevertheless have infinite global dimension. It remains an interesting question whether there are reasonable structural conditions guaranteeing the finite global dimension of a noetherian Hopf $\mathbb{K}$-algebra. For a significant necessary condition, see Theorem \ref{gldimthm}.
\end{remark}

Naturally, Goodearl and Zhang asked \cite[Question 0.3]{GZ10} whether every affine noetherian Hopf $\mathbb{K}$-algebra of $GK$-dimension 2 satisfies $(\ref{xxeqnGK2})$. It did not take long to discover that the answer is ``no!'':

\begin{theorem}\label{GK2more} ({\rm Wang, Zhang, Zhuang}, \cite[Theorem 0.1, Corollary 0.2]{WZZ13})
\begin{enumerate}
\item[(i)] There exists an infinite family of affine noetherain Hopf $\mathbb{K}$-algebra domains, labelled $B(n, p_1, \ldots , p_s, q, \alpha_1, \ldots , \alpha_s)$ with $\alpha_1, \ldots , \alpha_s$ distinct integers, which are pairwise distinct affine noetherian Hopf domains of $GK$-dimension 2, variants of the algebras (c)$(B)$ listed in Theorem \ref{GKtwothm}, but not isomorphic to those listed there.
\item[(ii)] $B := B(n, p_1, \ldots , p_s, q, \alpha_1, \ldots , \alpha_s)$ is generated by group-likes and skew primitives, hence is pointed.
\item[(iii)] For all the algebras $B$, $\mathrm{Ext}_B^1(\mathbb{K}, \mathbb{K})  = 0$.
\item[(iv)] Suppose that $H$ is an affine noetherian Hopf $\mathbb{K}$-algebra domain with $\mathrm{GKdim}(H) = 2$ which is generated by its group-likes and skew primitive elements, and which contains no Hopf subalgebra isomorphic to $A(1,q)$ for $q$ a primitive fifth or seventh root of unity in $\mathbb{K}$. Then $H$ is isomorphic to one of the algebras of Theorem \ref{GKtwothm} or Theorem \ref{GK2more}(i).
\end{enumerate}
\end{theorem}

Wang, Zhang and Zhuang showed that their algebras have the following properties:

\begin{theorem}(\cite[Theorem 0.3]{WZZ13})\label{GK2moreprops} Let $B$ be an algebra as defined in Theorem \ref{GK2more}.
\begin{enumerate}
\item[(i)] $B$ satisfies a polynomial identity, in fact is a finite module over its affine centre. 
\item[(ii)] $B$ is Auslander-Gorenstein, $AS$-Gorenstein and $GK$-Cohen Macaulay, with $\mathrm{injdim}(B) = 2$.
\item[(iii)] $\mathrm{gldim}(B) < \infty \Longleftrightarrow \mathrm{gldim}(B) = 2 \Longleftrightarrow s = 2 \textit{ and } \alpha_1 \neq \alpha_2.$
\end{enumerate}
\end{theorem}

While the classification results for Hopf $\mathbb{K}$-algebras of $GK$-dimension at most 2 are deep and impressive, many open questions still remain, even for these low dimensions. We list some of these questions here, and add afterwards some brief remarks.

\begin{questions}\label{GKlowqns} (i) Are there prime noetherian Hopf $\mathbb{K}$-algebras of $GK$-dimesnsion 1 in addition to those described above?
\medskip

\noindent(ii) Is every affine noetherian Hopf $\mathbb{K}$-algebra domain of $GK$-dimension 2 generated by its group-likes and skew-primitives?
\medskip

\noindent(iii) Let $H$ be an affine noetherian Hopf $\mathbb{K}$-algebra domain of $GK$-dimension 2 which is generated by group-likes and skew-primitives. Is $H$ listed in Theorems \ref{GKtwothm} or \ref{GK2more}?
\medskip 

\noindent(iv) What are the prime affine noetherian Hopf $\mathbb{K}$-algebras of $GK$-dimension two? What are the prime regular ones?
\medskip

\noindent (v) Can we describe the semiprime affine noetherian Hopf $\mathbb{K}$-algebras of $GK$-dimension at most 2? Even the regular ones? Even those of $GK$-dimension one?
\medskip

\noindent (vi) Are there no Hopf $\mathbb{K}$-algebras $H$ with $\mathrm{GKdim}(H) = \alpha$, with $2 < \alpha < 3$?
\medskip

\noindent (vii) What can be said regarding the above questions for base fields which are not algebraically closed of characteristic 0?
\end{questions} 

\medskip

\begin{remark}\label{GKlowrmk} 

Regarding Questions \ref{GKlowqns}(v) and (vi), note that Bergman's Gap Theorem \cite[Theorem 2.5]{KL00} guarantees that there is no $\mathbb{K}$-algebra $R$ with $1 < \mathrm{GKdim}(R) < 2$. But by \cite[Theorem 2.9]{KL00} there exists a $\mathbb{K}$-algebra $S$ with $\mathrm{GKdim}(S) = \beta$ for every real number $\beta > 2$.


\end{remark}

\bigskip
\section{Properties of noetherian $AS$-Gorenstein Hopf algebras}\label{xxsec4} In this section we review a number of consequences flowing from the possession by a noetherian Hopf algebra of the homological properties discussed above.

\subsection{The antipode (I)}\label{xxsubsec4.1}
Takeuchi \cite{T71} showed that in general the antipode $S$ of  a Hopf algebra $H$ is not bijective. Nevertheless it has been conjectured, for example by Skryabin \cite[p.623]{Sk06}, that the antipode of a right noetherian Hopf algebra is bijective; he proved in \cite[Corollaries 1,2]{Sk06} that $S$ is bijective when $H$ is semiprime right noetherian or $H$ is affine noetherian PI. In fact the $AS$-Gorenstein condition is also sufficient to ensure that $S$ is bijective:

\begin{theorem}\label{xxthmbijantipode}({\rm Lu, Oh, Wang, Yu}, \cite[Corollary 0.3]{LOWY18}) If $H$ is a noetherian $AS$-Gorenstein Hopf algebra then its antipode $S$ is bijective.
\end{theorem}

The case of Theorem \ref{xxthmbijantipode} where $H$ is $AS$-regular was obtained independently by Le Meur \cite[Proposition 3.4.2]{LM19}. Since the bijectivity of the antipode is not assumed in the proof \cite{WZ02} of Theorem \ref{xxthm3.3}(i), the latter result combined with Theorem \ref{xxthmbijantipode} gives another proof of Skryabin's result that the antipode of an affine noetherian PI Hopf algebra is bijective.

\subsection{The dualising complex}\label{xxsubsec4.2}
We recall some standard terminology; for details see \cite{BZ08, LWZ09, LOWY18, LM19}. Given a $\mathbb{K}$-algebra $A$, its enveloping algebra $A \otimes A^{op}$ is denoted by $A^e$. Recall that the noetherian Hopf algebra $H$ is said to satisfy the \emph{Van den Bergh condition} if $\mathrm{injdim}(H) = d < \infty$ and $\mathrm{Ext}^i_{H^e}(H, H^e) = \delta_{id}U$ for an invertible $H$-bimodule $U$. In this case $U$ is called the \emph{Van den Bergh dualising module} for $H$. Then $H$ is said to satisfy \emph{Van den Bergh duality} if it satisfies the Van den Bergh condition and the $H^e$-module $H$ has a finite resolution by finitely generated projective $H^e$-modules. In this case $H$ satisfies a form of (twisted) Poincar$\acute{e}$ duality by \cite[Theorem 1 and Erratum]{VdB98}.

Finally, $H$ is \emph{twisted Calabi–Yau} if it satisfies Van den Bergh duality with the dualising module being ${}^{\nu}H^1$ for some algebra automorphism $\nu$ of $H$, and is \emph{Calabi–Yau} if $\nu$ is an inner automorphism. In this case $\nu$ is called the \emph{Nakayama automorphism} of $H$ (and is determined only up to an inner automorphism). 

For the description of $\nu$ in the following theorem, recall that, given a character $\chi:H \longrightarrow \mathbb{K}$ of $H$,  the \emph{left winding automorphism} $\tau^{\ell}_{\chi}$ of $H$ is defined by $\tau^{\ell}_{\chi}(h) = \sum \chi(h_1)h_2$ for $h \in H$. The notation used below in defining the complex $R$ indicates that it is concentrated in degree $d$.  The following result originates with \cite[Proposition 4.5, Corollary 5.2]{BZ08}, but incorporates improvements obtained in \cite[Lemma 1.3]{RRZ14} for (b) and in \cite[Theorems 3.3,3.4]{LOWY18} for (a), which permit the dropping of the bijectivity hypothesis on $S$ (thanks to Theorem \ref{xxthmbijantipode} above), and which respectively show also the necessity of the $AS$-regular and $AS$-Gorenstein hypotheses. The fact that hypothesis (b)(iii) implies (c) is a consequence of the main result of \cite{VdB98}. Recall here the notation $\int^{\ell}_H$ from (\ref{integraleqn}), and let $HH^{\ast}(H, -)$ and $HH_{\ast}(H, -)$ denote Hichschild (co)homology.

\begin{theorem}\label{dualisingthm}(\cite{BZ08, LM19, LOWY18}) Let $H$  be a noetherian Hopf $\mathbb{K}$--algebra with antipode $S$. 
\begin{itemize}
\item[(a)] The following are equivalent:
\begin{enumerate}
\item[(i)] $H$ is $AS$-Gorenstein with $\mathrm{injdim}(H) = d$;
\item[(ii)] $H$ satisfies the Van den Bergh condition;
\item[(iii)] $H$ has a rigid dualising complex $R = U[s]$, where $U$ is an invertible $H$-bimodule and $s \in \mathbb{Z}$. 
\end{enumerate}
In this case  $s = d$ and $U = {}^{\nu}H^1$, where the Nakayama automorphism $\nu = S^2 \circ \tau^{\ell}_{\chi}$, where $\chi$ is the character of the right $H$-action on $\int^{\ell}_H$.
\item[(b)] The following are equivalent:
\begin{enumerate}
\item[(i)] $H$ is $AS$-regular;
\item[(ii)] $H$ is twisted Calabi-Yau;
\item[(iii)] $H$ satisfies Van den Bergh duality, with dualising module $U$ as in (a).
\end{enumerate}
If these equivalent conditions hold then $H$ is Calabi-Yau if and only if $\chi = \varepsilon$ and $S^2$ is inner.

\item[(c)] Assume that $H$ satisfies the equivalent conditions listed in (b), with $\mathrm{gldim}(H) = d$. Then for every $H$-bimodule $M$ and for all  $i = 0, \ldots , d$,
$$ HH^i (H,\,M) \; = \; HH_{d-i}(H, \, {}^{\nu^{-1}}M), $$
where $\nu$ is the Nakayama automorphism as specified in (a).

\end{itemize}
\end{theorem}

\bigskip

\subsection{The antipode (II)}\label{subsectantipode}

Not only does the achievement of the $AS$-Gorenstein condition by a noetherian Hopf algebra $H$ force its antipode $S$ to be bijective (as we have seen in Theorem \ref{xxthmbijantipode}), but in this case $S^4$ also satisfies a (currently weaker) version of Radford's formula \cite{R76} for finite dimensional $H$:

\begin{theorem}\label{xxthmSformula}(\cite[Corollary 4.6]{BZ08},\cite[Theorem 3.1]{LOWY18}) Suppose that the noetherian Hopf $\mathbb{K}$-algebra is $AS$-Gorenstein. Then
$$ S^4 \; = \; \gamma \circ \tau^{\ell}_{\chi}\circ (\tau^r_{\chi})^{-1}$$
where $\gamma$ is an inner automorphism of $H$ and (as in Theorem \ref{dualisingthm}(a)) $\chi$ is the character of the left integral of $H$.
\end{theorem}

The proof of Theorem \ref{xxthmSformula} in \cite{BZ08} used the following amusing trick: observe that the Nakayama automorphism does \emph{not} involve the coalgebra structure of $H$, so we can evaluate the Nakayama automorphism $\nu'$ of $H' := (H, \Delta^{op}, S^{-1}, \varepsilon)$ using again Theorem \ref{dualisingthm}(a), then equate $\nu$ with $\nu'$. The disadvantage of this approach is that equality - that is, uniqueness - holds only up to an inner automorphism. 
Radford shows that when $H$ is finite dimensional, $\gamma$ is conjugation by the grouplike element of $H$ corresponding to the left integral of $H^{\ast}$. So we restate here the obvious

\begin{question}(\cite[Question 4.6 ]{BZ08}) In Theorem \ref{xxthmSformula}, what is $\gamma$?
\end{question}

\medskip

\subsection{Goldie quotient rings, homogeneity, finite-dimensional Hopf subalgebras}\label{Goldie} Recall that if $R$ is a ring and $\delta$ is a dimension function such that $\delta (M) \in \mathbb{R}_{\geq 0}$ for all finitely generated $R$-modules $M$, then a finitely generated $R$-module $M$ is $\delta$-\emph{homogeneous} if $\delta (N) = \delta (M)$ for every non-zero submodule $N$ of $M$. 

\begin{theorem}\label{qringthm} Let $H$ be an affine noetherian Auslander Gorenstein and $GK$-Cohen Macaulay Hopf $\mathbb{K}$-algebra.
\begin{enumerate}
\item[(i)] $H$ is (left and right) $GK$-homogeneous.
\item[(ii)] $H$ has a quasi-Frobenius Goldie quotient ring.
\end{enumerate}
\end{theorem}

\begin{proof}(i) If $I$ is a non-zero right or left ideal of $H$ then $j(I) = 0$ by Definition \ref{xxdefn2.1}(v) of the function $j$, so $\mathrm{GKdim}(I) = \mathrm{GKdim}(H)$ thanks to the $GK$-Cohen Macaulay condition (\ref{xxeqn2.2}).
\medskip

\noindent(ii) That $H$ has an artinian Goldie quotient ring follows from part (i) and Remark \ref{xxrems2.3}(v), since  Levasseur's \cite[Theorem 5.3]{Lev92} ensures that $Q(A)$ exists and is artinian for any noetherian algebra $A$ with finite GK-dimension, provided $GK$-dimension is exact on $A$-modules and $A$ is $GK$-homogeneous. That the quotient ring is quasi-Frobenius is \cite[Corollary 6.2(1)]{ASZ98}.
\end{proof}

It seems reasonable to suspect that the presence of the $GK$-Cohen Macaulay property amongst the hypotheses of Theorem \ref{qringthm} is a red herring, since its role in the proof is simply to guarantee that the dimension function $\rho := d - j(\,-\,)$ has desirable properties - for example exactness and symmetry on bimodules - which it may well always possess. So, since the $GK$-Macaulay property is often absent for a noetherian Hopf algebra, even one of finite $GK$-dimension (see for instance Theorem \ref{xxthm3.3}(iv)), we ask: 

\begin{question}\label{qringqn} Does Theorem \ref{qringthm} remain true if the $GK$-Cohen Macaulay property is omitted?
\end{question}

\medskip

In view of Remark \ref{gldimrem} and the hope that Theorem \ref{xxthm3.9}(i) may be a harbinger of more widely valid structural consequences of regularity, here are a couple of observations.

\begin{theorem}\label{gldimthm}Let $H$ be a noetherian Hopf $\mathbb{K}$-algebra of finite global dimension.
\begin{enumerate}%
\item[(i)] Let $A$ be a finite dimensional Hopf subalgebra of $H$. Then $A$ is semisimple.
\item[(ii)] Suppose that $H$ is Auslander regular and $GK$-Cohen Macaulay. Then $H$ is semiprime.
\end{enumerate}
\end{theorem}

\begin{proof}(i) By \cite[Theorem 6.1(ii)]{Sk07} $H$ is a free left and right $A$-module. Hence, since $\mathrm{gldim}(H) < \infty$, we can restrict a finite projective $H$-resolution of ${}_{\varepsilon} \mathbb{K}$ to $A$, to deduce that $\mathrm{prdim}_A ({}_{\varepsilon} \mathbb{K}) < \infty$. Since $A$ is a Hopf algebra, Lemma \ref{xxlem2.4}(i) implies that $\mathrm{gldim}(A)$ is finite. However, $A$, being a Hopf algebra, is Frobenius by \cite{LS69}. Since Frobenius algebras of finite global dimension are semisimple (because their projective modules are injective), the result follows.

\medskip

\noindent(ii) By Theorem \ref{qringthm}(ii) $H$ has a quasi-Frobenius Goldie quotient algebra $Q$. Every simple (say, left) $Q$-module $V$ can be realised as $V = Q \otimes_H V_0$ for any non-zero $H$-submodule $V_0$ of $V$. Since $H$ is regular, there is a finite projective $H$-resolution $\mathcal{P}$ of $V_0$. Since $Q$ is an Ore localisation of $H$, $Q$ is a flat right $H$-module,  so that $Q \otimes_H \mathcal{P}$ is a finite $Q$-projective resolution of $V = Q \otimes_H V_0$. Hence $\mathrm{gldim}(Q) < \infty$. But this forces $Q$ to be senisimple artinian (since projective $Q$-modules are injective). Therefore $H$ is semiprime by Goldie's theorem \cite[Theorem 2.3.6]{McCR87}. \end{proof}

\begin{remarks}\label{gldimthmremks}(i) In fact (as was pointed out to me by Quanshui Wu), in the proof of Theorem \ref{gldimthm}(i) we only need to assume that $H$ is \emph{weakly finite}, rather than noetherian, since \cite[Theorem 6.1(ii)]{Sk07} is valid under this much weaker finiteness condition.

\medskip
\noindent(ii) If Question \ref{qringqn} has a positive answer then the GK-Cohen Macaulay hypothesis can be omitted from Theorem \ref{gldimthm}(ii). 
\end{remarks} 

\bigskip

\subsection{Representation theory of regular PI Hopf algebras}\label{xxsubsecPIrep} We assume for simplicity throughout this subsection that 
\begin{equation}\label{fieldgood} \textit{ the ground field }\mathbb{K} \textit{ is algebraically closed.}
\end{equation}
Suppose that $Z$ denotes the centre of the ring $R$ and that\footnote{The $\mathrm{PI-degree}$ of the prime affine noetherian PI $\mathbb{K}$-algebra $R$ with centre $Z$ is $\sqrt{\mathrm{dim}_{Q(Z)}(Q(R))}$.}
\begin{equation}\label{Rhypo} R \textit{ is a prime affine noetherian }\mathbb{K}\textit{-algebra, a finitely generated } Z\textit{-module, } \mathrm{PI-degree}(R) = n.
\end{equation} 
So $Z$ is an affine domain by the Artin-Tate lemma, \cite[I.13.4]{BG02}. 

Then the simple $R$-modules all have finite $\mathbb{K}$-dimension, with the maximal such dimension being the PI-degree $n$ of $R$, \cite[Theorem I.13.5]{BG02}.  Moreover, \emph{most} of the simple $R$-modules $V$ have $\mathrm{dim}_{\mathbb{K}}(V) = n$. To make this claim precise, first recall that each simple $R$-module $V$ is annihilated by a maximal ideal $\mathfrak{m}_V$ of $Z$, by  the Nullstellensatz, again part of \cite[Theorem I.13.5]{BG02}. And, conversely, for each maximal ideal $\mathfrak{m}$ of $Z$, $Z/\mathfrak{m} \cong \mathbb{K}$ and $R/\mathfrak{m}R$ is a non-zero finite dimensional $\mathbb{K}$-algebra, by (\ref{fieldgood}), (\ref{Rhypo}) and Nakayama's lemma. Then, by (for example) \cite[Theorem III.1.6]{BG02},  
\begin{align*}  \mathcal{A} \; &:= \; \{\mathfrak{m} \in \mathrm{Maxspec}(Z) \, : \exists\, V  \textit{ simple } R\textit{-module,}\; \mathfrak{m}V = 0, \mathrm{dim}_{\mathbb{K}}(V) = n\}\\
&= \; \{\mathfrak{m} \in \mathrm{Maxspec}(Z) \, : \; R/\mathfrak{m}R \cong M_n (\mathbb{K}) \}.
\end{align*}
Furthermore, by \cite[Theorem III.1.7]{BG02}, $\mathcal{A}$ is a dense open subset of $\mathrm{Maxspec}(Z)$, called the \emph{Azumaya locus} of $R$.

We say now that $R$ is \emph{height one Azumaya} if the proper closed subset $\mathrm{Maxspec}(Z) \setminus \mathcal{A}$ has codimension at least 2 in $\mathrm{Maxspec}(Z)$. Denote by $\mathcal{S}$ the \emph{singular locus} of $\mathrm{Maxspec}(Z)$; that is, 
$$\mathcal{S} \; := \; \{\mathfrak{m} \in \mathrm{Maxspec}(Z) \, : \, \mathrm{gldim}(Z_{\mathfrak{m}}) = \infty \}.$$
We can now state

\begin{theorem}\label{Azumthm}(\cite[Theorem 3.8]{BG97}, \cite[Theorem III.1.8]{BG02}) Let $\mathrm{K}, \, R$ and $Z$ satisfy (\ref{fieldgood}) and (\ref{Rhypo}), and keep the notation $\mathcal{A}$ and $\mathcal{S}$ as above. Suppose that $R$ is Auslander regular and $GK$-Cohen Macaulay. If $R$ is height one Azumaya, then
$$ \mathcal{A} \; = \; \mathrm{Maxspec}(Z) \setminus \mathcal{S}.$$ 
\end{theorem}

\begin{remarks}\label{Azumrems}(i) For any pair $Z \subseteq R$ satisfying (\ref{fieldgood}) and (\ref{Rhypo}) it is not hard to show that every Azumaya point of $\mathrm{Maxspec}(Z)$ is smooth; that is,
$$ \mathcal{A} \; \subseteq \; \mathrm{Maxspec}(Z) \setminus \mathcal{S};$$
thus the height one Azumaya condition is only needed for the reverse inclusion.

\medskip
\noindent (ii) There are large classes of Hopf $\mathbb{K}$-algebras to which Theorem \ref{Azumthm} applies: if $\mathrm{K}$ is algebraically closed and $H$ is a prime affine noetherian Hopf $\mathbb{K}$-algebra satisfying a polynomial identity and with $\mathrm{gldim}(H) < \infty$, then by Theorem \ref{xxthm3.9}(i) $H$ is a finite module over its centre, and is Auslander-regular and $GK$-Cohen Macaulay. The key issue then is whether $H$ is height one Azumaya.
This has been shown to be so in many cases - for  example, quantised enveloping algebras and quantised function algebras with deformation parameter a root of unity are dealt with in detailed discussions in \cite[$\S$4]{BG97}; see also \cite[$\S$III.8.5]{BG02}. 

\medskip
\noindent (iii) The height one Azumaya hypothesis fails for some  prime regular PI Hopf algebras, and - with it - the conclusion of Theorem \ref{Azumthm}. Trivially, the hypothesis and the theorem both fail for any noncommutative regular prime Hopf algebra of $GK$-dimension one. More significantly, they fail also for $U(\mathfrak{sl}(2, \mathbb{K}))$ when $\mathrm{char}\mathbb{K} = 2$; see \cite[Example 3.4(ii), Theorem .104]{BG97} for details.

\medskip

\noindent (iv) One should bear in mind that, given any particular class $\mathcal{C}$ of affine noetherian PI Hopf algebras to which Theorem \ref{Azumthm} applies, that result is only the starting point in analysing the finite dimensional representation theory of the algebras in $\mathcal{C}$. For many such classes $\mathcal{C}$ - for instance enveloping algebras of finite dimensional Lie algebras in positive characteristic -  there is for each $H \in \mathcal{C}$ a canonical central Hopf subalgebra $C_H$ of $H$, with $H$ a free $C_H$-module of finite rank $n$. One can then analyse the finite dimensional representation theory of $H$ by studying the bundle of $n$-dimensional algebras $H/\mathfrak{m}H$ as $\mathfrak{m}$ ranges across $\mathrm{Maxspec}(C_H)$, a crucial case being the $n$-dimensional Hopf algebra $H/\mathfrak{m}_{\varepsilon}H$, where $\mathfrak{m}_{\varepsilon}$ is the augmentation ideal of $H$. 

For many important such classes $\mathcal{C}$ there is an underlying Poisson structure $\mathcal{P}$ on the algebra $C_H$, and it transpires that the isomorphism class of the algebra $H/\mathfrak{m}H$ is constant for maximal ideals $\mathfrak{m}$ belonging to the same symplectic leaf or - more generally - containing the same Poisson core. The interested reader should consult \cite{DCP93, BGo03} and some of the many later papers using these ideas.
\end{remarks}

\medskip

A natural question arising from Theorem \ref{Azumthm} and Remarks \ref{Azumrems} is: how does one go about calculating the Azumaya locus $\mathcal{A}$ of a given algebra $R$, given that the centre $Z(R)$ and/or its singular locus may be hard to determine? An answer that has proved successful in many cases is the use of discriminant ideal(s), developed in recent years by Yakimov and his collaborators. We sketch here a simplified version of one such result, \cite[Main Theorem]{BY18}. Moreover, as we indicate, this tactic of using discriminant ideals yields information also about the simple modules whose dimension is less than the PI-degree. Write, for $\mathfrak{m} \in \mathrm{Maxspec}(Z)$, 
$$\mathrm{Sim}_{\mathfrak{m}}(R) \; = \; \{V \textit{ simple } R\textit{-module }: \mathfrak{m}V = 0\}, $$
Since $\mathrm{dim}_{\mathbb{K}}(R/\mathfrak{m}R)$ is finite, $\mathrm{Sim}_{\mathfrak{m}}(R)$ is a finite set.

\begin{theorem}\label{BYak2018thm} Let $\mathbb{K}$, $R$ and $Z$ be as in (\ref{fieldgood}) and (\ref{Rhypo}), and assume that $Z$ is normal. Let $\mathrm{tr}$ be $\mathrm{tr}_{red}$ or $\mathrm{tr}_{st}$, respectively the reduced or the standard trace map from $R$ to $Z$.  
\begin{enumerate}
\item[(i)] The complement of the Azumaya locus of $R$ is the zero set of the ideal
$$ D_{n^2}(R/Z, \mathrm{tr}) \; := \; \langle \mathrm{det}\left[ \mathrm{tr}(y_i y_j \right]_{i,j = 1}^{n^2} \rangle , $$
as $(y_1, \ldots , y_{n^2})$ ranges through all $n^2$-tuples in $R^{n^2}$.
\item[(ii)] Let $\ell$ be a positive integer, and define $D_{\ell}(R/Z, \mathrm{tr})$ analogously to the case $\ell = n^2$. If $\ell > n^2$ then $D_{\ell}(R/Z, \mathrm{tr}) = 0$. If $\ell < n^2$ then the zero set $\mathcal{V}_{\ell}$ of $D_{\ell}(R/Z, \mathrm{tr}_{red})$ is
$$ \mathcal{V}_{\ell} \; = \; \{\mathfrak{m} \in \mathrm{Maxspec}(Z) \, : \, \sum_{V \in \mathrm{Sim}_{\mathfrak{m}}(R)}(\mathrm{dim}_{\mathbb{K}}(V))^2 < \ell \}. $$
\end{enumerate}
\end{theorem}

Notice that if $R = H$ is a prime regular affine noetherian PI Hopf $\mathbb{K}$-algebra then it satsifies the hypotheses of the above result by Theorem \ref{xxthm3.9}(i). Exploring the consequences of the above result and others related to it for the representation theory of PI Hopf algebras has been and remains currently a very active field, too rich to attempt to describe here. We limit ourselves to listing a few relevant papers which the interested reader may follow up - see for example \cite{NTY17, BY24, MWY25, HMQW26}.

\bigskip

\section*{Acknowledgement} I thank Quanshui Wu and James Zhang for very helpful comments.

\end{document}